\documentclass[twoside,a4paper,11pt]{article}
\usepackage{amsfonts, amsbsy, amsmath, amsthm, amssymb, latexsym, verbatim, enumerate}

\usepackage{latexsym}
\newtheorem{propo}{Proposition}[section]

\newtheorem{lemma}[propo]{Lemma}
\newtheorem{corol}[propo]{Corollary}

\newtheorem{theo}[propo]{Theorem}

\newcommand{\ld}{,\ldots ,}

\newcommand{\ra}{ \rightarrow }

\newcommand{\lan}{ \langle }
\newcommand{\ran}{ \rangle }

\newcommand{\diag}{\mathop{\rm diag}\nolimits}

\newcommand{\Id}{\mathop{\rm Id}\nolimits}
\newcommand{\Irr}{\mathop{\rm Irr}\nolimits}

\newcommand{\al}{\alpha}

\newcommand{\lam}{\lambda }

\newcommand{\up}{^{-1}}

\newcommand{\el}{\end{lemma}}

\newcommand{\om}{\omega }

\def\d12{{_{12}}}

\def\acf{{algebraically closed field }}

\def\f{{following }}

\def\ii{{if and only if }}

\def\ir{{irreducible }}

\def\irt{{irreducible. }}
\def\irr{{irreducible representation }}

\def\itf{{It follows that }}

\def\mult{{multiplicity }}

\def\rep{{representation }}
\def\reps{{representations }}

\def\syl{{Sylow $p$-subgroup }}

\date{} 
\begin{document}

\title{Irreducible representations of simple algebraic groups in which a unipotent element is 
represented by a matrix with single non-trivial Jordan block}
%Almost regular unipotent elements in representations of algebraic groupss}
\author{Donna Testerman and A.E. Zalesski\footnote{Acknowledgement. A part of this work was carried out with the generous support of the Bernoulli Center, at the 
Swiss Federal Institute of Technology Lausanne,
 when the second author participated in the workshop ``Local representation theory and simple groups'' (2016).}}
\maketitle

%file: DZ-part-3-2016-12-10

 \section{Introduction}

The \rep theory of algebraic groups is based on the study of weight spaces, which are
 nothing other than the
 homogeneous components  
with respect to  a maximal torus. As every semisimple element belongs to a maximal torus, 
the knowledge of 
weights and their multiplicities, 
in a given representation, yields rich information on eigenspaces of the element under consideration. 

It is much harder to obtain information on properties of unipotent elements, for example, 
their fixed point space,
their minimal polynomial, or in the best case, their 
Jordan block structure in a given representation. The situation is better for certain classes of 
elements,
 such as root elements, 
but in general, problems of this kind can be difficult.    

One such question was raised several years ago by the second  author, specifically:

Determine the \ir \reps $\phi$ of a simple algebraic group $G$ such that, for some unipotent element $u$, 
the Jordan normal form of $\phi(u)$ has exactly one  block of size greater than 1?

The main motivation for considering this question is to supply an additional tool for  the  recognition 
of linear groups via  properties of a single element. However, one can also view this question as a test of 
how well the general theory is adapted for solving computational problems on unipotent elements.

%Surprizingly, this question  turns out to be not easy. 
The first contribution was made by I. Suprunenko in 
\cite[Theorem 1.9]{S95}, who solved the problem in the case where $\phi(u)$ has exactly one Jordan block. 
Later she obtained a solution to the above problem for classical groups \cite[Theorem 3]{S2} 
(see \cite{S} for the proof). The current manuscript grew out of our work on overgroups of regular elements in simple algebraic groups (see \cite{TZ1,TZ2}). At that time, Suprunenko 
had announced a result
which can be used for solving  the above question for elements of order $p$ in the exceptional groups other than $G_2$; see Remark after Theorem~\ref{id5} for details. 
We have 
recently learned
that David Craven is working on similar questions for finite simple groups and their automorphism groups.

Our main result answers the above  question by considering all unipotent elements in 
all simple algebraic groups
of  exceptional type.  

\begin{theo}\label{mth1} Let G be  a simply connected simple linear algebraic group of exceptional Lie type over 
an algebraically closed field $F$
of characteristic $p\geq 0$, and let $u\in G$ be a nonidentity unipotent element. Let $\phi$ be a non-trivial 
irreducible representation of G. % with highest weight $\om$.
 Then the Jordan normal form of $\phi(u)$ contains at most one non-trivial block if and only if
 $G$ is of type $G_2$, $u$ is a regular unipotent element and $\dim \phi\leq 7$. 

\end{theo}

Theorem \ref{mth1} remains true when replacing $G$ by a finite quasi-simple group of Lie type, as every 
\ir $F$-representation of such a group lifts to a representation of an appropriate simple algebraic group.

Our method is different from those used in \cite{S95,S} and in a sense is indirect. 
We first consider the case where $G$ is of type $A_1$, and for a representation $\rho$, not necessarily 
irreducible,  we prove that the condition that $\rho(G)$ contains a unipotent element with only 
one non-trivial Jordan block implies that all non-zero weights of the representation are of \mult 1. 
Then we consider a special case where  
 $p=0$ or $|u|=p$,  and use a result of \cite{Te2,PST} saying that, with the exception of one class of elements 
in $G_2$, when $p=3$,  $u$  is contained in a simple  
algebraic 
subgroup of $G$  of type $A_1$. This implies that all non-zero weights of $\phi$ are of \mult 1. The \ir 
representations with this property are determined in \cite{TZ2}; the list is very short for $G$ of exceptional 
type.
 The Jordan normal form of all classes of unipotent elements in these representations was
 computed by Lawther \cite{La}. 
This yields the result for  $p=0$ or $|u|=p$. If $|u|>p>0$ then there is a suitable parabolic 
subgroup
 $P$ such that $u\notin R_u(P)$, the unipotent radical of $P$. So the projection $u'$ into a 
Levi subgroup $L$ of $P$ is non-trivial. Then 
 one can observe that ${\rm Jor}(\tau(u'))$ has single non-trivial block for every composition factor $\tau$ of
 the
 restriction 
 of $\phi$ to $L$. This allows us to use induction on the rank of $G$.

\medskip

{\it Notation} Throughout the paper $p$ denotes a prime or 0, and $F$ an algebraically closed field of 
characteristic $p$. Unless otherwise stated, $G$ is a simple simply connected algebraic group over $F$. 
All representations of 
$G$ and $FG$-modules are rational. To say that a representation $\rho$ of $G$ or an $FG$-module $M$ is 
irreducible,
 we write $\rho\in\Irr G$ or $M\in\Irr G$. We let $\{\alpha_1,\dots,\alpha_n\}$ be a base of the root system of $G$.
Our labelling of Dynkin diagrams is as in 
\cite{Bo}.

For an integer $n>0$ we denote by $J_n$ the Jordan block of size $n$, that is, the $(n\times n)$-matrix
with 1 at the positions $(i,i)$ and $(i,i+1)$ for $i=1\ld n$, and 0 elsewhere.  The Jordan block $J_1$ is called 
\emph{trivial}.  For a matrix $x$ we denote by ${\rm Jor (x)}$ a Jordan normal form of $x$. If $x$ is a 
linear transformation of a vector space $V$ we write 
${\rm Jor}_V (x)$ for a Jordan normal form of $x$, especially when we need to specify $V$. A diagonal matrix with
 diagonal entries $x_1\ld x_n$ is denoted by $\diag(x_1\ld x_n)$. A similar notation is used for a block-diagonal 
matrix.

%$$$$$$$

\section{Preliminaries}

In Lemma \ref{bl2} below $\rho_S^{reg}$ denotes the $FS$-module afforded by the regular representation of a 
finite group $S$.

 \begin{lemma}\label{bl2} Let $F$ be an \acf of
 characteristic $p>0$, let
 $G$ be a finite group with  \syl $S$ of order p, and
 let $M$ be an indecomposable $FG$-module. Suppose that $N_G(S)/S$
 is abelian.
 Then there is an indecomposable $FS$-module $L$ such that
 $M|_S=\frac{\dim M-\dim L}{|S|}\cdot \rho^{reg}_S\oplus L$ and $\dim L<p$.\end{lemma}

\begin{proof} If the restriction $M|_S$ is a projective $FS$-module then the statement is obvious  with $L=0$. 
Suppose otherwise. Set $N=N_G(S)$. By \cite[\S 19, Theorem 1]{AL}, 
 %\cite[Lemma VII.1.5]{Fe}, 
 $M|_N=L \oplus P$ %A_1\oplus A_2$ 
where $P$ is projective, %$A_2|_S$ is projective 
and $L $ is indecomposable. 
%the Green correspondent of $M$. Therefore, $(A_1\oplus A_2)|_S$ isprojective. 
As  $S$ is cyclic, every projective
$FS$-module is free, so $P|_S=\frac{\dim M -\dim
L}{|S|}\cdot \rho_S^{reg}$. Recall that $L$ is 
%indecomposable%(\cite[Theorem III.5.6]{Fe}) and 
uniserial (\cite[Theorem
VII.2.4]{Fe}), that is, the submodule lattice of   $L$  is a
chain. Let $S=\langle y\rangle $, and set $x=1-y$ in the group algebra $FN$, $L_0=L$ and
$L_i=x^iL$ for $i=1,\ldots ,d$ assuming $L_{d}=0$, $L_{d-1}\neq 0
$. So $d\leq p$. %To justify the claim on the dimension of the
%indecomposable components of $L|_S$ it suffices to show that the
%dimension of every  $L_i/L_{i+1}$ is the same for $i=0,1,\ldots
%,k-1 $. 
Observe that $L_1$ is an $FN$-module. (Indeed, for $n\in N
$ we have $nL_1=(1-nyn^{-1})L=(1-y^j)L$ for some integer $j>0$ and
$1-y^j=(1-y)+(1-y)y+\cdots +(1-y)y^{j-1}$.) Therefore, $L_i$ is 
an $FN$-module  for every $i$. As $S$ acts trivially on every
$L_i/L_{i+1}$, it is completely reducible as $FN$-module. Since $L$ is uniserial,
every $L_i/L_{i+1}$ is irreducible. Since $N/S$ is abelian, $\dim (L_i/L_{i+1})=1$.
%By \cite[Theorem VII.2.4]{Fe},
%all composition factors of  $L$ are of the same dimension $c$,
%say. Then $L |_S$ is a direct sum of $c$ copies of an
%indecomposable \rep of $S$ and   
So $d= \dim L $.  Here $d<|S|$ as
otherwise $L|_S$ is free and hence  so is $M|_S$. This completes the proof.\end{proof}

\begin{lemma}\label{ae1}   Let $J_m\in GL_m(F),J_n\in  GL_n(F)$, $1<n<m$ be Jordan blocks of size $m,n$
respectively. Then the Jordan form of $J_m\otimes J_n$ contains at least two blocks of size greater than $1$ 
unless  $m=n=2$ and $p\neq 2$.
\end{lemma}

\begin{proof} Let $X$ be a cyclic $p$-group if $F$ is a field of characteristic $p>0$, otherwise an infinite 
cyclic group. Let $ V_m,V_n$ be indecomposable $FX$-modules of dimensions $m,n$, respectively.
Let $V_i\subset V_m,$ $V_j\subset V_m$ be submodules of dimensions $i,j$, respectively. Then $V_i\otimes V_j $ is a 
submodule of $V_m\otimes V_n $. The number of indecomposable summands of an $FX$-module $M$ of dimension $\geq k$ 
is not less than that on any submodule of
$M$. It follows that the result follows by induction as soon as one verifies this  for $(m ,n)=(2,2),(3,2)$.

If $m=n=2$ then $V_2\otimes V_2 =W_1\oplus W_2$, where the pair $(\dim W_1,\dim W_2)$ is $(2,2)$ if $p=2$ and $(3,1)$ if 
$p\neq 2$. (This is well known. If $p\neq 2$, see \cite[Ch, VII, Theorem 2.7]{Fe}. If $p=2$
then $V_2$ is free, and hence so is  $V_2\otimes V_2$.) By induction, the lemma is true for $p=2$.

Let $m=3,n=2$. If $p=3$ then $V_3$ is free, and hence so is  $V_3\otimes V_2$. If $p\neq 2,3$ then
the lemma again follows by \cite[Ch, VII, Theorem 2.7]{Fe}.\end{proof}

\begin{corol}\label{c21} Let $G$ be an algebraic group and $u\in G$  unipotent.
Let M be an irreducible
FG-module such that ${\rm Jor}_{M}(u)$ has a single non-trivial block. Then either
M is tensor-indecomposable or $G=A_1$ and $\dim M=4$. 
\end{corol}

We will require the following generation result, due to Guralnick and Saxl.

\begin{lemma}\label{gs51}   {\rm \cite[Theorems 5.1 and 5.4]{GuSa}}  Let $G$ be an exceptional finite 
group of Lie type, of
 untwisted rank $l$, and $x\in (G\setminus Z(G)) $.  Then $G$ can be generated by $l+3$ conjugates of $x$, except, 
possibly, for the case $G=F_4$, $q$ even, $x^2=1$, where  $G$ can be generated by $8$ conjugates of $x$.
\end{lemma}

\begin{lemma}\label{gb3}  Let $G$ be an \ir subgroup of $GL_n(F)$ and $g\in G$. 
For an eigenvalue $\lam$ of $g$ set $d=\dim (\Id -\lam\up g)V$. Suppose that $G$ is   
generated by $m$ conjugates of $g$. Then $n\leq dm$. 

In addition, if G is an exceptional group of Lie type, of untwisted rank $l$, 
then $n\leq d(l+3)$,
 except, possibly, for  $G$  of type $F_4$, $q$ even,  $x^2=1$, where  $n\leq 8d$.
\end{lemma}

\begin{proof} Let $G=\lan g_1\ld g_m\ran$, where $g_i$ $(1\leq i\leq m)$ is conjugate to $g$ in $G$. Set 
$V'=\sum_{i=1}^m (\Id -\lam \up g_i)V$. Then $\dim V'\leq md$ and $GV'=V'$, whence $V=V'$, and the first statement 
follows. 

If $G$ is a finite exceptional group of Lie type then the additional statement follows from Lemma \ref{gs51}. 

\end{proof} 

\section{Some representations of groups $SL_2(p)$ and $SL_2(F)$}

\begin{lemma}\label{ae3}   Let $D=SL(2,p)\subset G=SL(2,F)$, $u\in D$ a
unipotent element and let $K$ be a tensor-decomposable \ir 
$FG$-module. Suppose that  ${\rm Jor}_K(u)$  contains a single non-trivial   block.  Then 
$p>2$ and $\dim K=4$. 
 In addition, $K|_D$ has a composition factor
 of dimension $3$, and $u$ has a block of size $3$.
\end{lemma}

\begin{proof} 
Let  $K=K_1\otimes K_2$, where $K_1$ is a
tensor-indecomposable $FG$-module and $d:=\dim K_1>1$.
By Lemma \ref{ae1},  ${\rm Jor}_{K_1}(u)$  and ${\rm Jor}_{K_2}(u)$
consist of blocks of size at most 2. As $K_1$ is irreducible and tensor-indecomposable,
 ${\rm Jor}_{K_1}(u)$   consists of  a single Jordan block. Therefore, $\dim K_1=2$ and $p\neq 2$. Obviously,  ${\rm Jor}_{K_2}(u)$  cannot have more than one block.  \itf  $K_2$ is tensor-indecomposable, and again by
Lemma~\ref{ae1}, $\dim K_2 = 2$.
 %Then  $\dim K_2=2$ as above, as claimed. In addition, $p>2$, otherwise $K_1$, and hence $ K=K_1\otimes K_2$ is a projective $D$-module of dimension 4, so ${\rm Jor}_{K}(u)$ has two blocks of size 2.
%The additional claim  follows from Lemma \ref{ae1}. 
As $K_1|_D\cong
K_2|_D$,  $K|_D$ contains as a direct summand the adjoint $FG$-module,
which is of dimension 3 for $p>2$.
\end{proof}

\medskip
The \f result is well known (see e.g. Humphreys \cite[12.4]{Hub}):

\begin{lemma}\label{t11}  Let $E$ be an indecomposable
 rational
module of composition length $2$ for a simple algebraic group.  Let
$\mu,\mu'$ be the highest weights of $E/L$, $L$, resp., where $L$
is the maximal submodule of  $E$. Then either $\mu <\mu'$ or $\mu
>\mu' $, and in the latter case  $E$ is of shape $W_\mu/M$,
where $W_\mu$ is the Weyl module of highest weight $\mu$ and $M$
is a submodule of $W_\mu$.
\end{lemma}

 \begin{corol}\label{a1b} Let $G=A_1$ and let $V$
be an $FG$-module. Suppose that $W$ is either a Weyl module or
indecomposable of composition length at most $2$. Then all weights
of $W$ are of \mult $1$. 
\end{corol}

\begin{proof}  If $p=0$ then all weights of an \ir $FG$-module are well known to be of \mult 1, and hence 
so are the weights of any Weyl module of $G$ for any $p>0$. If $W$ is an indecomposable of composition length 2 
then, by Lemma  \ref{t11}, either $W$ or the dual of $W$ is a quotient of a Weyl module, whence the claim.
\end{proof}

\begin{lemma}\label{aj1} {\rm \cite[Corollary 3.9]{AJL}}
Let $G=A_1$ and $V_{a\om_1}, V_{b\om_1}$ be \ir $FG$-modules
of highest weights $a\om_1,b\om_1$, respectively. Let $a=\sum _{i\geq 0}a_ip^i$ and
 $b=\sum _{i\geq 0}b_ip^i$ be the $p$-adic expansions of $a$ and $b$, respectively.
  Let $v_p(a+1)$ denote the maximum $r$ such that $p^r|a+1$. 
Suppose that there exists an indecomposable  $FG$-module of composition length $2$ with
factors $V_{a\om_1}$ and $ V_{b\om_1}$.Then there exists a natural number $k\geq v_p(a+1)$ such that $a_i=b_i$ for
$i\neq k,k+1$, and $a_k=p-b_k-2$ and $a_{k+1}=b_{k+1}\pm 1$. In particular, either $a\geq p$ or $b\geq p$. 
\end{lemma}

\begin{corol}\label{aj2}  
Let  $p>3$ and let $G,a,b$ be as
in Lemma {\rm \ref{aj1}}. Let $E$ be an $FG$-module with composition factors
$V_{a\om_1}$ and $ V_{b\om_1}$.
%\medskip $(1)$  
Suppose that $a=p^i+p^j$ and $b=p^r+p^t$ where
$i<j$, $r<t$.  Then $E$ is completely reducible.

\end{corol}

\begin{proof} Suppose the contrary.   Note that $a+1$ is coprime to $p$ as $p>2$. So
$v_p(a+1)=0$. %, that is,
%Lemma \ref{aj1} makes  no restriction to $k$ in our case.
We can assume (by swapping the
modules) that $i\leq r$. Suppose that $i< r$. Then $b_i=0$ and
$a_i=1$; by Lemma \ref{aj1}, $a_i=p-2$, which is false as $p>3$.
So $i=r$.    Then $j\neq t$, and we can assume $j<t$. Then $a_{j}=1, b_{j}=0$,
 and, by Lemma \ref{aj1}, $1=a_j=p-2$, a contradiction. 
\end{proof}

\noindent{\bf{Remark.}}The assumption on $a,b$ in Corollary \ref{aj2} is equivalent to saying that $V_{a\om_1}$ and $ V_{b\om_1}$ are tensor-decomposable and  $\dim V_{a\om_1}=\dim V_{b\om_1}=4$.

\begin{lemma}\label{ae4}  Let $D=SL(2,p)$ 
and let $S$
 be an indecomposable $FD$-module. Let $u\in D$, $o(u)=p$ and
 suppose that  ${\rm Jor}_{S}(u)$ contains a single non-trivial   block. Then
 $\dim S\leq p+1$, and
  one of the \f holds (where $l$ is the composition length of $S$):

\medskip
$(1)$ $l=1$  and $\dim S\leq p;$

\medskip
$(2)$  $l=2$, $p>2$ and $\dim S\in\{p-1,p+1\}$ or $p=2$ and $\dim S=2;$

\medskip
$(3)$   $l=3$, $p>3$,    $\dim S=p+1$ and
the dimensions of the composition factors of $S$ are $2,p-3,2 ;$

\medskip
$(4)$ $S$ has a  composition factor of dimension $p-2$ and all  other factors are trivial.

\medskip
\noindent In addition, if $\dim S\geq p$ then ${\rm Jor}_{S}(u)$ contains a %most one
 block of size $p$.

\end{lemma}

\begin{proof} The first claim
is well known if $S$ is \irt Suppose $S$ is reducible, and set $U=\lan u\ran$.
%If $S$ is projective then $S|_U$ is a free module (since $U$ is cyclic).
%So  $\dim S=p$ in this case by assumption. Suppose $S$ is not projective. Then, 
By
Lemma \ref{bl2}, the Jordan form of $u$ is $(m\cdot J_p,J_d)$ for some $d<p$.
By assumption,
$0\leq m\leq 1$, and   $m=1$ implies $d\leq 1$. Therefore,
$\dim S\leq p+1$. The additional claim (after item (4)) is obvious.

Consider the options for $l$.
If $l=1$  then the dimension of $S$ is well known to be at most $p$.
Suppose that $l>1$ and $p=2$. Then $D\cong SL_2(2)$, and  the non-trivial composition factors of $S$ are 
projective $D$-modules.
So either $S$ is \ir or trivial on the subgroup of $D$ of order 3. Then $\dim S=2$.
Let $p>2$.  If $l=2$ then  $\dim S=p-1$ or $p+1$, see \cite[p. 49]{AJL} or \cite[p.111]{Hub}.    
If $l\geq 3$ then (3) and (4) follow by applying (2) to the factors of $S$ of composition length 2.  
Indeed, let $T$ be an indecomposable submodule of $S$ of composition
length 2. By (2), $\dim T=p\pm 1$, and $\dim S\leq p+1$ by the above. So $\dim T=p- 1$, and $m:=\dim S/T\leq 2$. 
Let $d,e$ be the dimensions of the composition  factors of $T$, where $d+e=p-1$. As $S$ is indecomposable, 
there is an indecomposable quotient of $S$ of dimension $d+m$ or $e+m$. By (2),
$d+m$ or $e+m$ equals $p-1$. We may assume that $d+m=p-1$ (by reordering $d,e$). If $m=2$ then $d=p-3$ and 
$e=2$ so (3) holds. Here $p\neq 3$ as $d\neq 0$. If $m=1$ then $d=p-2$ and $e=1$, that is (4) holds.
\end{proof}

\medskip
The \f fact is trivial but it is convenient to state it explicitly as 
this is frequently used. 

\begin{lemma}\label{tr8} Let $M$ be an $FG$-module, and 
  $u\in G$   unipotent.
Suppose that  ${\rm Jor}_M(u)$  contains a single non-trivial  block.  Then the Jordan form of $u$ on any 
submodule or  quotient module
of $M$ contains at most one non-trivial block. The same is true for every
quotient $M_2/M_1$, where $M_1\subset M_2$ are $FG$-submodules of $M$.\end{lemma}

\begin{proof}  Indeed, $u$ has a single block of size $k>1$ on $M$ \ii the module $(u-1)M$ is uniserial as an 
$F\langle u \rangle $-module.
  This property is inherited by submodules. Applying this to the dual of  $M$, we get the result for quotient 
modules.
  These also imply the result for  $M_2/M_1$.\end{proof}

\begin{lemma}\label{bb1} Let 
$ G$ be of type $ A_1$, and let  $u\in G$ be a
unipotent element. Let $M$ be an $FG$-module and $M_0$ the maximal trivial submodule of $M$.
% Let $M_0$ be the maximal trivial $G$-submodule of$M$.
Suppose that  ${\rm Jor}_M(u)$  contains a single non-trivial   block.  Then the composition series of 
$M$ contains at most two non-trivial terms. More precisely, one of the \f holds:

\medskip
 $(A)$ the composition series of $M$ contains at most one non-trivial term; or

 %$C_H(T)$ contains a (root) unipotent element or

 \medskip
$(B)$ $p>2$, the composition length of $M/M_0$ is $2$ and   $p+1\leq \dim M/M_0\leq p+2$.

 \medskip

\noindent Moreover, in case (B), we have:\medskip

$(C)$ If $a\om_1,b\om_1$ are the highest weights of the composition factors of $M/M_0$, with $a\geq b$,
then $a\geq p$.

 %there are at most two weights of \mult $2$ and
% all %but one %other
%non-zero $G$-weights of $M$ are one-dimensional.
\end{lemma}

\begin{proof} For $p=0$ the lemma is trivial. So we assume $p>0$. 
  Obviously, we may assume that $M$ is indecomposable. %Set $M:=V |_G$.
 Let $D$ denote the subgroup of $G$
 isomorphic to $SL_2(p)$. Then $M|_D=S\oplus T$, where
 $T$ is a trivial $D$-module and $S$ is an indecomposable one.
  We first prove $(A)$ and $(B)$, in a sequence of steps (1) to (11).

\medskip
\noindent (1)  Every submodule of $M/M_0$ is indecomposable. In particular, the socle
 of $M/M_0$ is irreducible.

\medskip
Indeed, if $L$ is a submodule of $M/M_0$ and $L=L_1\oplus L_2$,
where $L_1,L_2$ are non-zero $FG$-modules, then one of them is trivial by Lemma \ref{tr8},
which contradicts Lemma~\ref{aj1} and the definition of $M_0$.

\medskip
\noindent (2) Let $M_1\subset M_2$ be $FG$-submodules of $M$.
Suppose that    ${\rm Jor}_{M_2/M_1}(u)$ has a block of size
  $p$. Then $M_1\subseteq M_0$ and $M/M_2$ is a trivial $FG$-module.

\medskip
Indeed,  $M_2/M_1$ has an indecomposable
$F\lan u\ran$-submodule $X$ of dimension $p$. Hence $X$ is projective and injective,
so $M/M_1|_{\langle u \rangle}=X\oplus Y$, where $Y$ is an
$F\lan u\ran$-module. By Lemma \ref{tr8}, $Y$ is a trivial $F\lan u\ran$-module.  As $X\subset
M_2/M_1$,  it follows that $u$ is trivial on $M/M_2$, and hence $M/M_2$ is a
trivial $FG$-module. Applying this to the dual of $M$, we observe
that $u$ acts trivially on $M_1$. So the claim follows.

%Set $K=M_2/M_1$ and $L=M_3/M_1$. Suppose the contrary. Then it suffices to handle the case where $\dim L/K=

\medskip
\noindent (3) Let $K$ be a  composition factor of $M$. %The lemma is true if
If  ${\rm Jor}_K(u)$  has a block of size $p$ then statement (A) holds.

\medskip
This follows from (2).

\medskip
\noindent (4) %The lemma is true i
If $p=2$ then the statement (A) holds.

\medskip
Indeed, in this case $|u|=2$ and $M$ has a non-trivial composition factor $K$, say. Then $u$ must have
a block of size 2 on $K$, so the result is true by (3).

\medskip
\centerline{\it{From now on we assume $p>2$.}}

\medskip
\noindent (5) %The lemma is true i
If $M$ has  a composition factor $K$ of dimension $p$ then  (A) holds.

\medskip
Since $\dim K=p$,  $K$ is tensor-indecomposable, and hence $K|_D$ is \irt
Then it is a projective $D$-module and  ${\rm Jor}_K(u)$  consists of a single block
 of size $p$. So the claim follows from (3).

\medskip
\noindent (6) Let $M_1\subset M_2\subset M_3$ be $FG$-submodules of $M$ such that $M_2/M_1$
is \ir and  $M_3/M_1$ is indecomposable. If $M_3/M_2$ is  trivial then (A) holds.

\medskip
Set $K=M_2/M_1$ and $L=M_3/M_1$. Suppose the contrary. Then it suffices to handle the case where $\dim (L/K)=1$.
If $\dim
K\leq p-1$ then $\dim L\leq p$. By \cite[Theorem 2]{Mn}, $L$ is completely reducible, contrary to the assumption.
 % which violates  Corollary \ref{aj2}(2). 
   So $\dim K>p$ by (5).
Therefore, $K$ is tensor-decomposable. By   Lemma \ref{ae3}, $\dim
K=4$ and  $p>3$ by Lemma \ref{ae3} and (3).   Then $\dim L=5\leq p$,
so $L$ is decomposable by \cite[Theorem 2]{Mn}. 
This is a contradiction.

\medskip
\noindent (7) Either statement (A) holds or all composition factors of $M/M_0$ are non-trivial and $M/M_0$ is uniserial.

\medskip
If $M/M_0$ has a trivial composition factor then the claim follows from (6). Otherwise, this follows from 
 Lemma \ref{tr8} and (1).

\medskip
\noindent (8) Suppose that     $M$  has a tensor-decomposable composition factor $K$, say.
Then either (A) holds and $\dim K=4$ or    $p>3$ and $(B)$ holds. %$\dim M/M_0=p+2$. %then the statement (A) holds.

\medskip
 By Lemmas \ref{tr8} and \ref{ae3},  $\dim K=4$,  the  composition factors of $K|_D$ are of dimensions $1,3$ and  
 ${\rm Jor}_K(u)$
 has a block of size 3. 
%If $p=2$ then (5) implies (A).
Suppose that (A) is false.
Then $p>3$ by (3).  %(as $u$ has no block of size $p$ on $K$ by (3)).
 Furthermore,  $ K$ can be included in a subquotient  $L$, say,  of composition length 2. Let $K'$ be the second factor of $L$.  By (7), $K'$ is non-trivial and $L$ is indecomposable. 
 
 Suppose first that $K'$ is tensor-decomposable. Then $\dim K'=4$ and, by Lemma \ref{aj2}, $L$ is completely reducible, which is false.

% as above, $\dim K'=4$, so $\dim L=8$. Then $p\leq 7
 %$ by \cite[Theorem 2]{Mn}.  Let $S$ be a non-trivial composition factor of the restriction $L|_D$.   $p=5$, and 

% Note that the highest weight of $K$ is of shape
 %$p^i+p^j$ for some $j>i\geq 0$.  Suppose that $K'$
%is tensor-decomposable. Then the highest weight of $K'$ is of shape $p^k+p^l$ and $l>k\geq 0$.
%By Lemma \ref{aj1}, $L$ is completely reducible, which is false

So $K'$ is tensor indecomposable,
and hence $K'|_D $ is \irt Set $m=\dim K'$, so $\dim L=4+m$.
Then  $1<m\leq p$, and
   $m<p$ by (5). Then $L|_D$  has composition factors of dimensions $3,1,m$, and
hence is decomposable (otherwise contradicts Lemma \ref{ae4}(3)).  As   ${\rm Jor}_L(u)$  has a single non-trivial  block,
it follows %by assumption 
that $L|_D$
contains an indecomposable  submodule $X$, say, with composition
factors of dimensions $3,m$. By Lemma \ref{ae4}, we have $3+m=p-1$
or $p+1$. In the former case $\dim L=p$, which is false in view of \cite[Theorem 2]{Mn}.
 So $3+m=p+1$, and hence $\dim L=p+2$. Furthermore,
by Lemma \ref{ae4},   ${\rm Jor}_X(u)$  contains a
block of size $p$. Let $L=M_2/M _1$ for some $FG$-modules $M_1\subset M_2$.
Then, by (2), $M/M_2$ and $M_1$ are  trivial $FG$-modules. So we deduce that $M_1=M_0$ and $M=M_2$, i.e. $L=M/M_0$, 
so  (B) follows.

\medskip
\noindent (9)   If the restriction $(M/M_0)|_D$ has a trivial composition factor then $(A)$ or $(B)$ holds.

\medskip
Suppose the contrary. Then by (7), all composition factors of $M/M_0$ are non-trivial, and tensor-indecomposable 
factors remain \ir upon restriction to $D$. So  one of  the  composition factors of $M$  is  tensor-decomposable,
which contradicts (8).

\medskip
\noindent (10) Either $(A)$ or $(B)$ holds,  or the restriction $(M/M_0)|_D$ is indecomposable and has no trivial
composition factor.

\medskip
Suppose that neither $(A)$ nor $(B)$ holds. Then, by (9),  $(M/M_0)|_D$    has no trivial
composition factor. By Lemma \ref{ae3}, every composition factor of   $M/M_0$ is tensor-indecomposable, and hence
irreducible for $D$. Then $(M/M_0)|_D$ is indecomposable in view of  Lemma \ref{tr8}.

\medskip
\noindent (11)  Statement (A) or (B) holds.    
%$p+1\leq\dim (M/M_0)\leq p+2$. 
%and  if $p=3$ then $\dim (M/M_0)=4$.

\medskip
Suppose the contrary. Then, by (8),  %$M$ has a tensor-decomposable composition factor then,  by (8),  $(A)$ or $(B)$ holds.
%reasoning in (8) shows that either $p=3$ and 
%$\dim (M/M_0)=4$ or $p>3$ and $\dim (M/M_0)=p+2$. 
%So suppose that
 the composition factors of $ M$ are tensor-indecomposable and hence are irreducible for $D$. By (10),
 $(M/M_0)|_D$ is indecomposable with no trivial composition factor.  Then, by Lemma \ref{ae4},  $p>2$ and $\dim (M/M_0)\leq p+1$.
If  $\dim (M/M_0)\leq p$ then $M/M_0$ is completely reducible by \cite{Mn},
which contradicts (1). % unless $(A)$ holds.

So %either (A) or (B) holds, or 
we have $p > 2$ and $\dim(M/M_0) = p+1$, %Moreover, we may
and all composition factors are %tensor indecomposable and so
 irreducible for $D$. 
If case (3) of
Lemma \ref{ae4} holds, then $ M/M_0$ has composition length 3 with tensor indecomposable
factors of dimension 2, $p - 3$, 2, contradicting Lemma \ref{aj1}. So Lemma \ref{ae4}(2) must hold and
$M/M_0$ has composition length 2 as in (B).

%Lemma \ref{tr8} and the definition of $M_0$. This completes the proof of (A) and (B).

Finally, statement (C) follows from (B) and Lemma \ref{aj1}. 
%Suppose the contrary, and let $a<p$. By Corollary \ref{a1b}, all non-zero weights of $M/M_0$ are of \mult 1. 
%If $a-b$ is even then weight $b\om_1$ occurs in $V_{a\om_1}$, so $a-b$ is odd. In this case the central involution of $SL_2(F)$ acts non-trivially on $M/M_0$, and then $M/M_0$ is decomposable, contrary to (1).
\end{proof}

\begin{lemma}\label{aa12}  Let $ G\cong A_1$ and let $M$ be an $FG$-module.
Let  $u\in G$ be a unipotent element.
Suppose that   ${\rm Jor}_M(u)$  contains a single non-trivial block.
Then all non-zero weights of $M$, with respect to a fixed maximal torus of $G$, are of \mult $1$.
\end{lemma}

 \begin{proof} For $p=0$ the lemma is trivial, for $p>0$  this follows from Lemma \ref{bb1} and Corollary \ref{a1b}.\end{proof} 

\begin{lemma}\label{aa1}   Let
$ G$ be a simple algebraic group and $X\cong A_1$ a subgroup of $G$.
Let  $u\in X$ be a unipotent element and  $M$ an $FG$-module.
Suppose that  ${\rm Jor}_M(u)$  contains a single non-trivial  block.   Then %one of the \f holds:
%\medskip
% $(a)$ $C_H(T)$ contains a (root) unipotent element or\medskip
%$(b)$ %there are at most two weights of \mult $2$ and
%either all %but one %other
all non-zero weights of $M$, with respect to a fixed maximal torus of $G$, are of \mult $1$. Moreover $M$
 is tensor-indecomposable, unless $p\neq 2$, $G=A_1$
 and $\dim M=4$. 
\end{lemma}

\begin{proof} Suppose the contrary, and fix a maximal torus $T$ of $G$, and a maximal torus $T_1$ of $X$ with 
$T_1\subset T$. Let $M_\lam$ be a $T$-weight space of weight $\lam\neq 0$ such that $\dim M_\lam>1$. Then 
$\dim M_{w(\lam)}>1$ for every $w\in W$, where $W$ is the Weyl group of $G$. By  Lemma \ref{aa12}, $T_1$ acts 
trivially on $M_{w(\lam)}$ for every $w\in W$. Recall that the weights of $G$ are elements of 
${\rm Hom}(T, GL_1(F))$, which is  a ${\mathbb Z}$-lattice of rank $r$ equal to the rank of $G$. 
The Weyl group acts on $T$ and
 hence on ${\rm Hom}(T, GL_1(F))$, so $W$ is realized as a subgroup of $GL_r({\mathbb Z})$.   Let 
$R$ be the vector space over the rational number field ${\mathbb Q}$   spanned by the weights, and
 this yields an embedding  of  $W$ into $GL_r({\mathbb Q})$. It is well known that $W$ is an \ir subgroup of 
$GL_r({\mathbb Q})$.  
 The subspace of $R$ spanned by $\{w(\lam):w\in W\}$ is $W$-stable, and hence coincides with $R$.
 Therefore, every weight $\mu$ can be written as $\sum_{w\in W} a_w w(\lam)$ with $a_w\in {\mathbb Q}$.
 Let $m$ be an integer such that $ma_w\in {\mathbb Z}$ for every $w\in W$. Then 
$m\mu=\sum_{w\in W} (ma_w) w(\lam)$, where the coefficients $ma_w$ are   integers.  This 
implies that $(m\mu)(T_1)=1$,
 whence $\mu(T_1^m)=1$. 
 Note that for every $t_1\in T_1$ there is an element  $t\in T_1$ such that $t^m=t_1$, 
in other words $T_1=T_1^m$. Therefore, $\mu(T_1)=1$. This is true for every weight $\mu$ of $T$. 
This implies that $T_1$ acts trivially on $M$, which is a contradiction. 
 
 For the second assertion in the lemma see Corollary \ref{c21}. \end{proof}

\begin{theo}\label{fr55}  
 Let $ G$ be a simple algebraic group, % and $X\cong A_1$ a subgroup of $G$.
  $u\in G$  a unipotent element and  $M$ an $FG$-module.
Suppose that   ${\rm Jor}_M(u)$  contains a single non-trivial  block.  If  $p>0$ and $u^p=1$, or if $p=0$, then 
 all  non-zero weights of $M$, with respect to a fixed maximal torus of $G$, are of multiplicity 1,
 unless possibly $G=G_2,p=3$ and $u$ lies in the class $A_1^{(3)}$
as in {\rm \cite{PST}}.
\end{theo}

\begin{proof}   By the main results of \cite{Te2,PST},   every element of order $p$ in a simple algebraic group  
in defining characteristic $p$ is contained in a simple algebraic subgroup of type $A_1$, with the exception of the class of elements labelled $A_1^{(3)}$ in $G=G_2$ when    
$p= 3$. If $p=0$, every unipotent element is well known to lie in 
a subgroup of type $A_1$. So the statement follows from Lemma \ref{aa1}.  \end{proof} 
%{fr5}. 

\section{Representations of groups of exceptional type}

In view of Theorem~\ref{fr55}, it is useful to know which irreducible representations of exceptional 
algebraic groups have all non-zero weights of multiplicity 1. Moreover, for our application
to the question about the Jordan block structure of unipotent elements in the representation
space,  Corollary~\ref{c21} shows that we can restrict our attention to tensor-indecomposable representations. We have the following result taken from {\rm \cite{TZ2}}.

 \begin{lemma}\label{fr5} Let $ G$ be a simple algebraic group of exceptional type and  
let $M$ be a tensor-indecomposable \ir $FG$-module with  highest weight $\om\neq 0$. Suppose that all non-zero weights of 
$M$ are of \mult one. Then $(G,\om)\in\{(E_6,p^a\om_i), i=1,2,6, (E_7,p^a\om_j), j=1,7, (E_8,p^a\om_8), (F_4,p^a\om_k), k=1,4,
(G_2,p^a\om_l), l=1,2\}$, for some $a\geq 0$.
\end{lemma}

 \begin{theo}\label{id5} 
 Let $G$ be a simple algebraic group of exceptional type and  
let  $1\ne u\in G$ be a unipotent element of order $p$. Let $M$ be an \ir $FG$-module.
Then one of the following holds:
\begin{enumerate}[]
\item{\rm (a)} ${\rm Jor}_M(u)$  contains at least two non-trivial blocks;
\item{\rm (b)} $G=G_2$, $\omega=p^a\omega_1$, $p\geq 7$, $u$ is regular and ${\rm Jor}_M(u)$ has precisely one block;
\item{\rm (c)} $G=G_2,p=3$ and $u$ lies in the class labelled $A_1^{(3)}$. 
\end{enumerate}
\end{theo}

\begin{proof}  By Corollary~\ref{c21}, $M$ is tensor indecomposable. As $|u|=p$, Theorem \ref{fr55} leaves us 
with inspection of the unipotent block structure for the  cases
listed in Lemma \ref{fr5}, unless  we are in the situation of (c) above.
The Jordan block structure of all unipotent elements in the representations  listed in Lemma \ref{fr5} has been computed by Lawther \cite{La}. We see that either (a) or (b) holds.
(Note that the case $|u|=2$ can be deduced from classification   of irreducible linear  groups generated by transvections, see for instance \cite{Po}.)
 \end{proof}

\noindent{\bf{Remark.}}   Under the assumptions of Theorem~\ref{id5} and assuming in addition that 
$G\neq G_2$, a result of Suprunenko \cite[Theorem 1]{S05} allows one to reduce the 
problem under discussion to an analysis of modules $M$ of dimension at most $4(l+3)$ and a small list
of further exceptions, most of which can be handled using the Tables of \cite{La}.
However, the proof of Suprunenko's result 
announced in 2005 has not been yet published, so we prefer to avoid using it. In addition, 
the method based on Theorem \ref{fr55} can also lead to an alternative proof of a similar theorem for 
classical algebraic groups.

 \begin{comment}
\begin{lemma}\label{up1}  Let $P$ be a parabolic subgroup  of $G$ with Levi subgroup $L$. 
Suppose  $L'$  has roots of the distinct  lengths if so is $G$.  Let $u$ be a unipotent element of G.  Then there is a conjugate of u
that does not lie in $R_u(P)$, the unipotent radical of P.
\end{lemma} 

\begin{proof}  Recall that the Weyl group permutes transitively the root subgroups of roots of the same length. Then every root subgroup is conjugate to one in $L'$.
This implies the claim in the case where $u$ is contained in a root subgroup. In general, express $u$ as a product of  elements
$u_i$, say, which lie in distinct simple root subgroups, and assume that   distinct $u_i$ are in distinct simple root subgroups.  Then we can assume that $u_1\in L'$,
and hence $u_1\notin R_u(P)$. We may assume that $P$ is a standard parabolic, that is, contains a maximal torus and all positive root subgroups, and $L'$ is generated by some of the above root subgroups.  Then either $u_i\in L'$ or $u_i\in R_u(P)$.
This implies the  lemma. \end{proof} 
  \end{comment}
  
  \medskip

We wish to extend Theorem \ref{id5} to the case $|u|>p$. 
For this we use induction. From Lemma  \ref{tr8} we get the following

\begin{lemma}\label{cv6}  
Let $P$ be  a parabolic subgroup in $G$ with Levi factor $L$ and
unipotent radical $U.$ Let $u\in P$ be unipotent and let $u'$ denote the
projection of $u$ into $L$. Let $V$ be an \ir $FG$-module and
$0=V_0\subset V_1\subset V_2\subset \cdots\subset V_k=V  $ be a
socle series  of $V|_P$ (that is, $U$ acts trivially on the factors
$V_i/V_{i-1}$). Suppose that the Jordan form of $u$ on $V$ has at most
one block of  size greater than $1$. Then the same holds for the Jordan form of
$u'$ on every factor $V_i/V_{i-1}$.
\end{lemma}

\begin{lemma}\label{as1}  Let $G$ be a simple algebraic group in characteristic  $ p>0$, 
$P$ a parabolic subgroup and $U=R_u(P)$  the unipotent radical of $P$. If $U$ is abelian then the exponent of $U$
 equals  
$p$. 
If  $U$ is nilpotent of class $2$ then $U$ is of exponent $p$ or $4$. 

%Assume $p>2$ for $G$ of type $B_n,C_n,F_4$ and $p>3$ for $G_2$.
%Then U is of exponent p
\end{lemma}

\begin{proof} It is well known that $U$ is generated by root subgroups $U_\al$ for some sets of roots $\al$, and 
each $U_\al$ is an abelian group of exponent $p$. This implies the statement if   $U$ is abelian.   Otherwise,  
let $U'$ be the derived subgroup of $U$. Then for every $x\in U$ the mapping    $u\ra [x,u]$ ($u\in U$) yields a 
group homomorphism $U\ra U'$. Now $U'$ is of exponent $p$ as so is $U/U'$. If  $p=2$ then  $U$ is of exponent 4.
If $p>2$ then
$(xu)^p=x^pu^p[u,x]^{p(p-1)/2}=x^pu^p$ as $[u,x]^{p(p-1)/2}=1$.   This easily implies  the lemma. \end{proof} 
%n the result follows from \cite[Theorem 1]{abs}.  

\medskip
\noindent{\bf{Remark.}} Below we will apply Lemma \ref{as1} to groups $G$ of type $E_6$ or $F_4$,  and  to the maximal parabolic
subgroups $P$ corresponding to nodes $1,6$ for $E_6$ and $1,4$ for $F_4$.  Then $U$ is of nilpotency class 2.
This follows from the description of $U$ in \cite[4.4]{cks} for node 1 of $E_6$, and for node 6 this
follows too as the graph automorphism of $E_6$ permutes the nodes 1,6. For $G=F_4$ this similarly follows from 
 \cite[4.5]{cks} for node 1, and for node 4 from data at \cite[p. 19]{cks}. (More precisely, 
 $U$ is generated by the root subgroups $U_\al$, where $\al$ runs over  positive roots whose expression in terms of simple roots contains $\al_4$, and $U$ contains a normal subgroup $R$ generated by $U_\al$
 for $\al$  such that the root $\al_4$ occurs in such expressions with coefficient 2. As no positive root of $F_4$
 has $\al_4$-coefficient greater than 2, the claim follows from \cite[4.8(i)]{cks}.)

\begin{lemma}\label{a44}  Let G be a simple exceptional algebraic group of rank l in defining characteristic $p$, 
and $u\in G$ unipotent.
Let $M$ be an \ir   G-module such that ${\rm Jor}_M(u)$ contains a single non-trivial block. 
%and  $M'$ such that g is almost cyclic 
Let $k\geq 0$ be an integer such that $|u^{p^k}|=p$. Then $\dim M\leq (p-1)p^k(l+3)$, unless possibly $G$ is of 
type $F_4$, $p=2$, where $\dim M\leq 2^{k+3}$. 

% $g^4=1$.
%Let $M$ be an \ir G-module such that g is almost cyclic on M. Then $\dim M\leq 2(l+3)$, except  $G=F_4$ where $\dim M\leq 16$.  Moreover, if $G\neq G_2$ is of exceptional type then $\dim M=1$.
\end{lemma}

\begin{proof} %Let $G_1$ be a finite group of  the same type as $G$ containing $g$.
As ${\rm Jor}_M(u)$ contains a single non-trivial block, it follows that ${\rm Jor}_M(u^{p^k})$ contains at most $p^k$ 
 non-trivial blocks, each of size at most $p$. Therefore, $\dim (\Id -u^{p^k})M\leq (p-1)p^{k}$. There is a finite 
group $G_1\subset G$ of Lie type such that $u\in G_1$ and $M$ is an \ir $FG_1$-module. So the result follows by
 applying Lemma  \ref{gb3} to $u^{p^k}$.\end{proof} 

%\medskip
Example. If $|u|\leq 4$ then $\dim M\leq 2(l+3)$ or $16$ for $F_4$. If $|u|=9$ then $\dim M\leq 6(l+3)$.

\begin{lemma}\label{fe6}  Let G be of type  $E_6,E_7,E_8$ or $F_4$, and let $1\neq u\in G$ be a unipotent element.
 Let $V$ be an \ir G-module  such that $\dim V>1$. Then ${\rm Jor}_V(u)$ has at least two   non-trivial blocks.    
\end{lemma}

\begin{proof} By Lemma \ref{ae3}, we can assume that $V=V_\om$ is tensor-indecomposable, and hence 
without loss of generality the highest 
weight $\om$ of $V$ is $p$-restricted. In view of Theorem \ref{id5}  we can assume that $|u|>p$. If $|u|=4$
 then $\dim M\leq 22$ (see Example following Lemma \ref{a44});  however $G$ is well known to have no 
non-trivial \irr of degree less than 25. So we can assume that $|u|>2p$. Suppose the contrary, that 
${\rm Jor}_V(u)$ has a single non-trivial block.

Suppose first that $G$ is of type $E_6$. Let $P_i$, $i=1,6$, be a maximal parabolic of $G$ corresponding to nodes 
1, respectively 6, of the Dynkin diagram of $G$. Let $L_i$ be a Levi subgroup of $P_i$ and $L_i'$ the derived subgroup of $L_i$. Then $L'_i$ is a simple group of type $D_5$. 
By Lemma \ref{as1},  $u\notin  R_u(P)$.  Let $u'$ be the projection of $u$ into $L_i$. Then
 $1\neq u'\in L'_i$.  By Lemma \ref{cv6}, the  Jordan form of $u'$ has at most one  non-trivial block on
  every  composition factor of the restriction of $V$ to $L'_i$.
 Let $\lam$ be the   highest weight of a non-trivial composition factor.
  
Recall that    $L_1'$ is generated by the root subgroups $U_{\pm \al_i}$ with $i\in\{2, 3, 4,5,6 \}$, and $L_2'$ is 
generated by the  root subgroups $U_{\pm \al_i}$ with $i\in\{1,2,3,4,5\}$.
Let $\om=\sum a_i\om_i$, where $\om_1\ld \om_6$ are the fundamental weights of $E_6$ and $0\leq a_i<p$
for $i=1\ld 6$. Let $\lam_1\ld \lam_5$ be the fundamental weights of $D_5$. By Smith's theorem \cite{Sm}, 
the restriction of $V$ to $L_1'$ contains a composition factor
of highest weight $\sum_{i=1}^5a_i\lam_i$, and    the restriction of $V$ to $L_2'$ has a composition factor
of highest weight $ a_6\lam_1+a_5\lam_2+a_4\lam_3+a_3\lam_4+a_2\lam_5$. (The root ordering for $L_2'\cong D_5$ is 
inverse of that in $L_1'$.) By a result of Suprunenko \cite[Theorem 3]{S} for $p>2$ and \cite[Theorem 3]{S2} 
for $p=2$,  
 $\lam= p^m\om_1$ for some integer $m>0$. Applying this to $L_1'$, we obtain 
$a_1\leq 1$, $a_2=\cdots =a_5=0$, applying to $L_2'$, we get $a_6\leq 1$.  
So we  are left to examine the cases where $\om\in\{\om_1,\om_6,\om_1+\om_6\}$. By Lawther \cite[p. 4136]{La}, 
${\rm Jor}_V(u)$ has at least two non-trivial blocks for   $\om=\om_1$ and $\om_6$.

Let $\om=\om_1+\om_6$. Then $\omega-\alpha_1$, respectively $\omega-\alpha_6$, affords the highest weight of an 
$FL_i$-composition factor for $i=1$, resp. $i=2$,  with highest weight $\lambda_1+\lambda_4$.
%consider a parabolic subgroup $P_2$ with Levi subgroup $L_2'$ of type $A_5$
%generated by the root subgroups corresponding to the roots $\al_1,\al_3,\al_4,\al_5,\al_6$. Then, by Smith's theorem, \emph{loc.cit.}, the restriction of $V$ to $L_2'$ contains a composition factor
%with highest weight $\mu_1+\mu_5$, where $\mu_1\ld \mu_5$ are fundamental weights of $A_5$.
This again contradicts \cite[Theorem 3]{S} for $p>2$ and \cite[Theorem 3]{S2} for $p=2$.

If $G$ is of type $E_7$ then $G$ has a parabolic subgroup $P$ whose Levi factor $L'$ is  
of type $E_6$, and $R_u(P)$
is abelian (\cite[4.4]{cks}), and hence of exponent $p$.  As above we deduce that $u\not\in R_u(P)$. 
If $\omega=\omega_1$, $\omega$ affords an $L'$-composition factor of highest weight $\omega_1$ for $L'=E_6$, and 
if $\omega=\omega_7$, the weight $\omega-\alpha_7$ affords an $L'$-composition factor which is again one of the 
27-dimensional irreducible $E_6$-modules. But this contradicts the conclusion of the previous paragraph.  
In an entirely similar way, and using \cite[4.4]{cks}, the case $G=E_8$ follows from that for $E_7$.

Let $G=F_4$, 
and let $\om=\sum a_i\om_i$, where $\om_1\ld \om_4$ are the fundamental weights of $F_4$. Let $P_i$, $i=1,4$, be  a
 maximal parabolic of $G$ corresponding to nodes 1 or 4 of the Dynkin diagram of $G$. Let $L_i$ be a Levi subgroup 
of $P_i$ and $L_i'$ the derived subgroup of $L_i$. Then 
$L_4'$ is  simple of type $B_3$, and $L_1'$ is  simple of type $C_3$. Let  $\lam_1,\lam_2, \lam_3$ be the 
fundamental weights of $L_4'$ and $ \mu_1,\mu_2, \mu_3$  the fundamental weights of $L_1'$.
As  above,  by  Smith's theorem, the restriction of $V$ to $L_4'$ contains a composition factor
with highest weight $a_1\lam_1+a_2\lam_2+a_3\lam_3$, and    the restriction of $V$ to $L_1'$ has a composition factor
with highest weight $ a_4\mu_1+a_3\mu_2+a_2\mu_3$. Applying \cite[Theorem 3]{S}
to $L_4'$, we get $a_1=1, a_2=a_3=0$; applying to $L_2'$ \cite[Theorem 3]{S} for $p>2$ and \cite[Theorem 3]{S2} for 
$p=2$,
 we get 
 $ a_4\leq 1$. 
So we have to examine the cases $\om\in\{\om_1,\om_4, \om_1+\om_4\}$.  
By Lawther \cite[p. 4134, 4135]{La}, ${\rm Jor}_V(u)$ has at least two non-trivial blocks for   
$\om=\om_1$ and $\om_4$. If  $p=2$  and $\om=\om_1+\om_4$ then $V$ is tensor-decomposable (see \cite[Corollary
of Theorem 41]{St}), so   
${\rm Jor}_V(u)$ has at most two non-trivial blocks by Lemma \ref{c21}.

Let $\om=\om_1+\om_4$ and $p>2$. By L\"ubeck \cite{Lu}, $\dim V_{\om}=1053$. 
As $|u|>p>2$, we have $|u|\leq  3^3,5^2,7^2,11^2$ for $p=3,5,7,11$ respectively, see \cite{La}.  By Lemma \ref{a44},
we have $\dim M\leq 7\cdot 10\cdot 11=770$. This is a contradiction.  
 \end{proof}

%\medskip
%DONNA' TEXT: 
Now we  consider the remaining cases for the group $G =
G_2$. So from now on we assume $G =
G_2$, that is,
the unipotent elements which are either of order greater than $p$
or the one class of elements of order $3$ for $p=3$ which do not lie in any
$A_1$-type subgroup \cite[Theorem 5.1]{PST}. We fix a maximal torus $T$ of $G$ and root subgroups with respect to $T$. For all roots $\alpha$, let $x_\alpha:{\bf G}_a\to G$ be an isomorphism whose image is the $T$-root subgroup $U_\alpha$ corresponding to $\alpha$.
By \cite{La}, for example, we are left to consider the following:
\begin{enumerate}[a)]
\item $u$ is regular and hence conjugate to $x_{-\alpha_1}(1)x_{-\alpha_2}(1)$, $p\leq 5$, and
$u$ has order $p^2$ or 8. %+\delta_{2,p}}$.
\item $u$ is in the class $G_2(a_1)$, $p=2$, $u$ is conjugate to $x_{\alpha_2}(1)x_{3\alpha_1+\alpha_2}(1)$
and has order 4.
\item $u$ has order 3 %and %lies in the class $A_1^{(3)}$, 
and is conjugate to
$x_{2\alpha_1+\alpha_2}(1)x_{3\alpha_1+2\alpha_2}(1)$.
\end{enumerate}

We  note that the Jordan block structure of all unipotent
elements acting on the irreducible modules with highest weight
$\omega_1$, or $\om_2$ for $p\ne 3$, %, on the module $V_G(\omega_2)$,
is given in Lawther \cite{La}.  We use this to show:

\begin{lemma}\label{lawther_blocks} Let $1\ne u\in G=G_2$ be unipotent and let $V$ be one of the two %restricted
irreducible $FG$-modules with %fundamental dominant 
highest weight $\om_1$ or $\om_2$. 
Then ${\rm Jor}_V(u)$ has a single   
non-trivial block  % $u$ has a unique Jordan block of size greater than 1 on $V$
if and only if one of the following holds:
%\begin{enumerate}\item 

$(1)$ $u$ is regular and $\om=\omega_1$.

$(2)$ $u$ is regular, $p=3$ and $\om = \omega_2$.
%\end{enumerate}
\end{lemma}

\begin{proof} For the weight $\om=\om_2,p\neq 3$ and $\om=\om_1$
the statement follows directly from \cite[Table 1]{La}. % \emph{loc.cit}.
So it remains to consider the case of the irreducible module
$V_G(\omega_2)$, when $p=3$. We apply the exceptional graph
automorphism of $G$ and see that any element acting with only one
non-trivial Jordan block on $V_{\omega_2}$ must have image an
element acting with only one non-trivial Jordan block on
$V_{\omega_1}$. Then, by the above remarks, the image of the
element under the graph automorphism must be regular, which means
the element itself is regular. The result follows.
\end{proof}

\begin{propo}\label{gg22}  Let $G=G_2$ and let $u\in G$ be unipotent. Let V be an \ir
 G-module with highest weight $\om$. Then ${\rm Jor}_V(u)$ has a single %t most two non-trivial 
non-trivial block  %the Jordan form of $u$ on $V$ has exactly one block of size greater than one 
\ii
 $\om=p^k\om_1$ or $p=3$ and $\om=p^k\om_2$ for some integer $k\geq 0$. 
 
\end{propo}

\begin{proof} Suppose the contrary. By Corollary \ref{c21}, %it suffices to deal with the case where 
$V$ is tensor-indecomposable, % in particular,  
%In addition, 
so we may then assume that V is 
$p$-restricted. Note that $u$ is conjugate to a unipotent element of $G_2(p)$, and the restriction of a 
$p$-restricted \irr to $G_2(p)$ remains irreducible. So it suffices to deal with $G=G_2(p)$, which we 
assume in some cases below. 
By Lemma \ref{lawther_blocks}, we can assume that   $\om\neq \om_1,\om_2$. 

Let $p=2$. Then $|u|\leq 8$. Let $u_2$ be a power of $u$ such that $|u_2|=2$. 
 By Lemma \ref{a44}, we get  $\dim V\leq 20$. However, as   $\om$ is $2$-restricted and 
$\om\neq \om_1,\om_2$, we have $\om=\om_1+\om_2$. By \cite{Lu}, $\dim V=64 $ in this case, 
which is a 
contradiction.

 For elements of order $p$ for $p=3$, as in c) above, we similarly obtain the bound 
$\dim V\leq 10$, whereas the minimal dimension of an \irr of $G$ with highest weight 
$\om\neq 0,\om_1,\om_2$ exceeds
  26 \cite{luebeck}.  So we are left with the case $|u|=9$ for $p=3$ and $|u|=25$ for $p=5$. 
In these cases $u$ is a regular unipotent element of $G$, say $u=x_{-\alpha_1}(1)x_{-\alpha_2}(1)$.

Let $P_i\leq G$ be the parabolic subgroup with $P_i\supseteq B^-$, the Borel subgroup generated by the maximal 
torus $T$ and the 
root subgroups corresponding to negative roots, 
and whose Levi factor $L_i$ satisfies $L_i' = \langle
U_{\pm\alpha_i}\rangle$. Set $Q_i = R_u(P_i)$ and let
$\pi_i:P_i\to L_i$ be the canonical projection. Set $u_i=\pi_i(u)\in L_i'$.
 As $u\in P_i$,
$u$ stabilizes the commutator series $V\supset[V,Q_i]\supset [[V,Q_i], Q_i]\supset\cdots$, and acts on the 
quotients, via the element $u_i$. Then by Lemma \ref{tr8}, the matrix of $u$ in its 
action on every subquotient 
has at most one non-trivial Jordan block and, by Lemma \ref{bb1}(C), each of these subquotients can have at most 
one non-trivial
$p$-restricted \ir constituent. Note that we will abuse notation and write $\omega_i$ for the restriction of 
$\omega_i$ to $T\cap L_i'$.
Setting $[V,Q_i^0] = V$, let $[V,Q_i^d]=[[V,Q_i^{d-1}]Q_i]$ for $d\geq 1$; we will use the following result from 
\cite[2.3]{GS}:\medbreak

\noindent{\it{If $p=3$, assume $\omega=r\omega_1$, for some $r$.  Fix an integer $d\geq 0$.
Then the quotient $[V,Q_i^d]/[V,Q_i^{d+1}]$ is 
isomorphic 
to the direct sum of those weight spaces of $V$ of the form $\omega-d\alpha_j-m\alpha_i$, for some $m\geq 0$ and 
 where $\{i,j\}=\{1,2\}$.}}
\medbreak

\medskip
\noindent{\bf{Case $p=3$.}}  

Note that $V$ is tensor-indecomposable \ii $\om=2\om_1$ or $2\om_2$ (we ignore the cases $\om=\om_1,\om_2$ by the
 above), see \cite[Corollary of Theorem 41]{St}.

The class of regular elements is invariant under the graph
automorphism of $G$ and so the Jordan block structure of a
regular element on $V_{a\omega_1}$ is the same as the Jordan block
structure of this element on $V_{a\omega_2}$. So it suffices to deal with $\om=2\om_1$.
Using the above quoted result \cite[2.3]{GS}, we see that $[V,Q_1^1]/[V, Q_1^2]$   has three
$FL_1'$-composition factors, afforded by $\omega-\alpha_1-\alpha_2$ and $\omega-2\alpha_1-\alpha_2$, (the latter
weight has multiplicity 2, see \cite{luebeck}, and affords two composition factors) with highest weights 
$3\omega_1$, respectively $\omega_1$. 
This contradicts Lemma \ref{bb1}(G).

%So $\dim M=6>p+2$,
%and  Lemma \ref{bb1} yields a contradiction.  

 \medskip
\noindent{\bf {Case $p=5$.}}  

Throughout, we will refer to \cite{luebeck} for weight multiplicities, without further reference.

Consider first the modules $V = V_{a\omega_1}$, $a=2,3,4$. Here the $FL_1'$-module
$[V,Q_1]/[V,Q_1^2]$ has $FL_1$-composition
factors of highest weights $(a+1)\omega_1$ and $(a-1)\omega_1$, afforded by $\omega-\alpha_1-\alpha_2$, 
respectively $\omega_1-2\alpha_1-\alpha_2$. Then Lemma~\ref{bb1} implies that $a=4$. But in this case the second 
weight has multiplicity 2 and affords a third non-trivial composition factor, contradicting Lemma~\ref{bb1}(C).
%This contradicts Lemma \ref{bb1}.

Now turn to the modules whose highest weight is of the form
$b\omega_2$. For $V_{2\omega_2}$, the $FL_2'$-module
$[V,Q_2^3]/[V,Q_2^4]$ has composition
factors of highest weights $3\omega_2$, $\omega_2$ and $\omega_2$,
afforded by $\omega-3\alpha_1-\alpha_2$, respectively $\omega-3\alpha_1-2\alpha_2$, the latter having multiplicity 3 in $V$. This contradicts Lemma~\ref{bb1}(C).

 For $V_{3\omega_2}$, the
$FL_2'$-module $[V,Q_2^3]/[V,Q_2^4]$ has composition factors  of highest weights $4\omega_2$ and
$2\omega_2$, afforded by $\omega-3\alpha_1-\alpha_2$, respectively $\omega-3\alpha_1-2\alpha_2$, 
contradicting Lemma~\ref{bb1}(C).

Finally, for
the $FG$-module $V_{4\omega_2}$, we consider the action of the
parabolic subgroup $P_1$. The $FL_1'$-module $[V,Q_1^3]/[V,Q_1^4]$ has a
composition factor $R$ of dimension 10 whose highest weight is $9\omega_1$ (afforded by the weight 
$\omega-3\alpha_2$).
 Then $R$ is a tensor product of modules of dimensions 2 and 5,
 which contradicts Lemma \ref{ae1}.

We now turn to modules $V_{a\omega_1+b\omega_2}$, where
$0<a,b<5$. By \cite[1.35]{Te}, the weight $\omega-\alpha_1-\alpha_2$ has multiplicity 2 in $V$ if and only if
$(3b+a+3)\not\equiv 0\mod 5$. If $\omega-\alpha_1-\alpha_2$ has multiplicity 2, the $FL_i$-module 
$[V,Q_i]/[V,Q_i^2]$ has composition factors of highest weights $a+3$ and $a+1$, or $b+1$ and $b-1$, for $i=1$,
 respectively 2, afforded by $\omega-\alpha_j$ and $\omega-\alpha_j-\alpha_i$, where $\{i,j\}=\{1,2\}$. Now using
repeatedly Lemma~\ref{bb1} and Lemma~\ref{ae1}, we deduce that $\omega\in\{\omega_1+2\omega_2,
3\omega_1+b\omega_2 \ (b=1,3,4), 4\omega_1+\omega_2\}$. If $b>1$, the weight $\omega-2\alpha_2$ affords 
an $FL_1$-composition factor of $[V,Q_1^2]/[V,Q_1^3]$ which is tensor decomposable and contradicts Lemma~\ref{ae1}.
If $\omega=a\omega_1+\omega_2$, for $a=3,4$, then $\omega-3\alpha_1$ and $\omega-3\alpha_1-\alpha_2$ afford 
$FL_2$-composition factors of $[V,Q_2^3]/[V,Q_2^4]$, of highest weights $4\omega_2$, respectively $2\omega_2$, 
contradicting Lemma~\ref{bb1}(C).\end{proof}

This completes the consideration of the remaining cases for the group $G=G_2$ and together with Theorem~\ref{id5}
completes the proof of Theorem~\ref{mth1}.

\end{document}